\documentclass[12pt,twoside,letterpaper]{amsart}
\usepackage{amsmath,amssymb,fullpage}   
                            
\usepackage{url}            
\usepackage{graphicx}       
\usepackage{amsthm}

\numberwithin{equation}{subsection}

\theoremstyle{definition}
\newtheorem{theorem}[equation]{Theorem}
\newtheorem{lemma}[equation]{Lemma}
\newtheorem{proposition}[equation]{Proposition}
\newtheorem{corollary}[equation]{Corollary}

\newtheorem{definition}[equation]{Definition}

\newtheorem{remark}[equation]{Remark}

\newtheorem{claim}[equation]{Claim}

\newcommand{\ad}{\operatorname{ad}}

\newcommand{\Fg}{\mathcal{F}}

\newcommand{\HH}{\mathsf{HH}}
\newcommand{\HHH}{\mathcal{H}}
\newcommand{\HP}{\mathsf{HP}}

\newcommand{\ZZ}{\mathbb{Z}}

\newcommand{\PP}{\mathbb{P}}

\newcommand{\C}{\mathbb{C}}

\newcommand{\OO}{\mathcal{O}}

\newcommand{\Hom}{\text{Hom}}
\newcommand{\Ext}{\text{Ext}}

\newcommand{\g}{\mathfrak{g}}

\newcommand{\SL}{\mathsf{SL}}
\newcommand{\GL}{\mathsf{GL}}
\newcommand{\Sp}{\mathsf{Sp}}
\newcommand{\gr}{\operatorname{\mathsf{gr}}}
\newcommand{\Weyl}{\mathsf{Weyl}}
\newcommand{\onto}{\twoheadrightarrow}
\newcommand{\into}{\hookrightarrow}
\newcommand{\Sym}{\operatorname{\mathsf{Sym}}}
\newcommand{\ctr}{\lrcorner\,\,}
\newcommand{\rx}{\mathrm{x}}
\newcommand{\ry}{\mathrm{y}}

\begin{document}
\title{Zeroth Poisson homology of symmetric powers of isolated 
quasihomogeneous surface singularities}
\author{Pavel Etingof and Travis Schedler}
\date{July 9, 2009}
\begin{abstract}
  Let $X \subset \C^3$ be a surface with an isolated singularity at
  the origin, given by the equation $Q(x,y,z)=0$, where $Q$ is a
  weighted-homogeneous polynomial. In particular, this includes the
  Kleinian surfaces $X = \C^2/G$ for $G < \SL_2(\C)$ finite.  Let
  $Y := S^n X$ be the $n$-th symmetric power of $X$.  We compute the
  zeroth Poisson homology $\HP_0(Y)$, as a graded vector space with
  respect to the weight grading.  In the Kleinian case, this confirms
  a conjecture of Alev, that
  $\HP_0(\C^{2n}/(G^n \rtimes S_n)) \cong \HH_0(\Weyl_{2n}^{G^n
    \rtimes S_n})$,
  where $\Weyl_{2n}$ is the Weyl algebra on $2n$ generators. That is,
  the Brylinski spectral sequence degenerates in this case. In the
  elliptic case, this yields the zeroth Hochschild homology of
  symmetric powers of the elliptic algebras with three generators
  modulo their center, $A_\gamma$, for all but countably many
  parameters $\gamma$ in the elliptic curve.
\end{abstract}
\maketitle
\tableofcontents

\section{Introduction}
\subsection{Main result}\label{mrs}
Let $a, b, c$ be positive integers, and equip $\C[x,y,z]$ with a
weight grading in which $|x|=a, |y|=b$, and $|z|=c$.  In this paper,
we are interested in surfaces $X \subset \C^3$ with an isolated
singularity at the origin, cut out by a polynomial $Q(x,y,z) = 0$,
which is weighted-homogeneous of degree $d$.  Such surfaces were first
studied systematically by Saito \cite{Sa}. For convenience,
we also assume that $a \leq b \leq c$.

The surface $X$ is equipped with a standard Poisson bracket, given by
the bivector
\begin{equation}
\pi := (\frac{\partial}{\partial x} \wedge \frac{\partial}{\partial y} \wedge \frac{\partial}{\partial z}) \ctr (dQ),
\end{equation}
where $\ctr$ is the natural contraction operation, which in this case
produces  a bivector from a trivector and a one-form.  The above
bivector is, moreover, weight-homogeneous of degree $\kappa := d-(a+b+c)$, and
is a Poisson bivector (i.e., $\{\pi, \pi\} = 0$, where $\{\,, \}$ is
the Schouten-Nijenhuis bracket). Hence it produces a Poisson bracket
of degree $\kappa$.

In particular, when $\kappa < 0$, $X$ has a Kleinian singularity, i.e.,
$X \cong \C^2 / G$ where $G < \SL_2(\C)$ is a finite subgroup. These finite subgroups have a well-known classification, and up to equivalence, we have:
\begin{gather} \label{amdesc}
A_{m-1}: G = \ZZ/m, a=2, b=c=m, Q = x^m + y^2 + z^2, \\
D_{m+2}: G = \widetilde{D_{2m}}, a=2, b=m, c=m+1, Q = x^{m+1} + x y^2 + z^2, \\
E_6: G = \widetilde{A_4}, a=3, b=4, c=6,   Q = x^4 + y^3 + z^2, \\
E_7: G = \widetilde{S_4}, a=4, b=6, c=9,  Q = x^3 y + y^3 + z^2, \\
E_8: G = \widetilde{A_5}, a=6, b=10, c=15, Q = x^5 + y^3 + z^2. \label{e8desc}
\end{gather}
Here we set the degree $\kappa$ of the bracket to $-2$ in the $A$ case and $-1$ for the $D$ and $E$ cases.

The case $\kappa = 0$ (i.e., $d=a+b+c$) is called the \emph{elliptic}
case, and, up to equivalence, the surface has one of
the following forms, for some $\lambda \in \C^\times$:
\begin{gather}
\widetilde E_6: a=b=c=1, Q= x^3 + y^3 + z^3 + \lambda xyz, \\
\widetilde E_7: a=b=1, c=2, Q = x^4 + y^4 + z^2 + \lambda xyz,\\
\widetilde E_8: a=1, b=2, c=3, Q = x^6 + y^3 + z^2 + \lambda xyz.
\end{gather}

Let $X^{(n)} := S^n X$ be the $n$-th symmetric power of $X$, which is a
singular affine Poisson variety.
In this paper, we compute explicitly the zeroth Poisson homology of $X^{(n)}$, as a graded vector space using the weight grading.  To describe this, recall the Jacobi ring of $X$,
\begin{equation}
J_Q := \C[x,y,z] / (Q_x, Q_y, Q_z),
\end{equation}
where $Q_x, Q_y$, and $Q_z$ are the partial derivatives of $Q$ with
respect to $x, y,$ and $z$, respectively.
Then, $J_Q$ is finite-dimensional, and its dimension is called the
\emph{Milnor number}, and denoted by $\mu_Q$.

For any graded vector space $V$ with finite-dimensional graded
components, let $h(V;t)$ denote its Hilbert series.

Recall that, for any Poisson algebra $A$, its zeroth Poisson homology
is defined as
\begin{equation}
\HP_0(A) := A / \{A, A\}.
\end{equation}
In \cite{AL}, Alev and Lambre showed that
\begin{equation}
\HP_0(\OO_X) \cong J_Q,
\end{equation}
as weight-graded vector spaces.  Here $\OO_X$ denotes the global functions on $X$ (which is affine).

It will be convenient to combine the linear duals of the homology groups
$\HP_0(\OO_{X^{(n)}})$, for $n \geq 0$, into one bigraded algebra
$\bigoplus_{n \geq 0} \HP_0(\OO_{X^{(n)}})^*$, with the multiplication
given by the symmetrization maps
$\OO_{X^{(i)}}^* \otimes \OO_{X^{(j)}}^* \rightarrow \OO_{X^{(i+j)}}^*$. The main result
of this paper is
\begin{theorem} \label{mt} $\bigoplus_{n \geq 0} \HP_0(\OO_{X^{(n)}})^*$ is isomorphic, as a bigraded algebra, to a free commutative algebra generated by a bigraded vector space $L$ with Hilbert series
\begin{equation} \label{mtfla}
h(L;t^{-1},s) = \frac{h(J_Q;t)s}{1 - t^{d} s}.
\end{equation}
\end{theorem}
Here, the exponent of $t$ is the weight, and the
exponent of $s$ is the corresponding symmetric power of $X$.  Note
that there is a $t^{-1}$ since, by convention, the weights are negated
when we take the dual.

We may thus write the Hilbert series of
$\bigoplus_{n \geq 0} \HP_0(\OO_{X^{(n)}})$ itself by the following
formula. Write $h(J_Q;t) = t^{n_1} + \cdots + t^{n_r}$ (in the
Kleinian case, the numbers $m_i := n_{i}+1$ (in types $D,E$) or
$m_i := \frac{n_i}{2}+1$ (in type $A$) are the Coxeter exponents
associated to the corresponding finite Weyl group of type ADE, which has Coxeter number  $h = d$ (types $D, E$) or $h = \frac{d}{2}$ (type $A$)).
\begin{equation}
  h(\bigoplus_{n \geq 0} \HP_0(\OO_{X^{(n)}});t,s)= \prod_{i=1}^r \prod_{j \geq 0} \frac{1}{1-t^{n_i + jd} s^{j+1}}. 
\end{equation}
\subsection{Hochschild homology of deformations and Alev's conjecture}
In the Kleinian case, i.e., \eqref{amdesc}--\eqref{e8desc}, 
$X \cong \C^2/G$ for the listed finite subgroup
$G < \SL_2(\C)$. Using this, one has a canonical quantization of
$\OO_{X^{(n)}}$, namely $\Weyl_{2n}^{G^n \rtimes S_n} = \Sym^n \Weyl_2^{G}$, where
the Weyl algebras are defined as
$\Weyl_2 = \C \langle x,y \rangle / ([x,y]-1)$ and
$\Weyl_{2n} = \Weyl_2^{\otimes n}$, and $G^n \rtimes S_n$ is the
semidirect product where $S_n$ acts by permuting components (i.e., the
wreath product of $G$ with $S_n$), viewed as a subgroup of
$\Sp_{2n}(\C)$.  

Here, by ``quantization'' of a graded Poisson algebra $B$ with bracket
of degree $-f$ for some $f > 0$, we mean a filtered associative
algebra $A = \bigcup_m A_{\leq m}$ such that $\gr(A) = B$, and such
that, for every $a \in A_{\leq m}, b \in A_{\leq n}$, we have
$ab - ba \in A_{\leq(m+n-f)}$ and the image of $ab - ba$ in
$\gr_{m+n-f}(A)$ is $\{\gr_m(a), \gr_n(b)\}$.  (We could alternatively
define this as a deformation quantization satisfying a homogeneity
property.)

In this situation, there is a standard Brylinski spectral sequence
from $\HP_0(B)$ to $\HH_0(A) := A /[A,A]$. Moreover, we may equip
$\HH_0(A)$ with the weight filtration from $A$, and it is easy to see
that this spectral sequence preserves the grading, in the sense that
each page consists of homogeneous differentials. Thus, the spectral
sequence converges to $\gr \HH_0(A)$.

One may ask whether the Brylinski spectral sequence degenerates.  As a
consequence of our main theorem, we may deduce the
\begin{theorem}\label{kleindeg}
In the case that $B = \OO_{X^{(n)}}$ for $X = \C^2/G$ a Kleinian singularity listed in \eqref{amdesc}--\eqref{e8desc}, and $A = \Weyl_{2n}^{G^m \rtimes S_m}$, the Brylinski spectral sequence $\HP_0(B) \Rightarrow \gr \HH_0(A)$ degenerates, yielding an isomorphism of graded vector spaces, $\HP_0(B) \cong \gr \HH_0(A)$.
\end{theorem}
This confirms a conjecture of J. Alev \cite[Remark 40]{Bu}. In the
case where $G = \ZZ/2 = \{\pm \operatorname{Id}\} \subset \Sp_2(\C)$,
this was proved in the case $n = 2$ in \cite{AlFo}, and for $n=3$ in
\cite{Bu}, where also some preliminary results and conjectures are
given towards general $n$ (again for $G = \ZZ/2$). 

We remark that Alev also posed a similar conjecture, which replaces
$G^n \rtimes S_n$ as above with finite Weyl groups
$W < \GL_n(\C) \into \Sp_{2n}(\C)$, where the embedding is given by
$A \mapsto \begin{pmatrix} A & 0 \\ 0 & (A^{t})^{-1} \end{pmatrix}$.
The case $(\ZZ/2)^n \rtimes S_n$ then identifies with the Weyl groups
of type $B_n$ (and also with type $C_n$).  We do not know whether this
conjecture holds for general Weyl groups of types other than $B$ (or
$C$), although it was verified in types $D_2 = A_1 \times A_1$ and
$D_3 = A_3$ in \cite{Bu}.

To deduce Theorem \ref{kleindeg} from Theorem \ref{mt}, one uses
\cite{AFLS}, which gives a general formula for the dimension of
$\HH_0(\Weyl_{2n}^G)$ for arbitrary $n$ and $G < \Sp_{2n}(\C)$.
However, we will use only the $n=1$ case and a general result from
\cite[\S 3]{EO} to make this more transparent.

  In the non-Kleinian cases, the Poisson bracket does not have
  negative degree, so the above does not apply. However, following
  \cite{EGdelpezzo} (e.g., Theorems 3.4.4 and 3.4.5), one may always
  produce a deformation quantization of $\OO_X$, i.e, a
  $\C[[\hbar]]$-algebra $A_\hbar$ which is isomorphic to
  $\OO_X[[\hbar]]$ as a $\C[[\hbar]]$-module, such that
  $A_\hbar/(\hbar) \cong \OO_X$ and such that
  $[a,b] = \hbar \{a,b\} + O(\hbar^2)$ for all
  $a,b \in \OO_X \subset \OO_X[[\hbar]] \cong A_\hbar$. 
 For such a
  deformation quantization, we may similarly deduce the
  \begin{theorem} \label{defdeg} The Brylinski spectral sequence
    $\HP_0(\OO_{X^{(n)}})((\hbar)) \Rightarrow \gr \HH_0(\Sym^n
    A_\hbar[\hbar^{-1}])$ degenerates.
\end{theorem}
This generalizes the theorem above.   

In the case that the Poisson bracket has degree zero, there exist not
merely formal but actual, \emph{homogeneous} quantizations of $A$, the
Artin-Tate-Odesski-Sklyanin-type algebras $A_\gamma$ modulo their
center (e.g., \cite{ATV, Steph}; see also \cite[\S 3.5]{EGdelpezzo}). Here
the parameter $\hbar$ is replaced by a point $\gamma$ on an elliptic
curve. For such algebras, we deduce
\begin{corollary} \label{ellcor} For all but countably many parameters $\gamma$,
we have a noncanonical isomorphism of weight-graded vector spaces, 
$\HH_0(\Sym^n A_\gamma) \cong \HP_0(\OO_{X^{(n)}})$, and moreover,
\begin{equation} \label{ellcorfla}
\bigoplus_n \HH_0(\Sym^n A_\gamma)^* \cong \Sym(\HH_0(A_\gamma)^*[t]),
\end{equation}
as bigraded algebras (noncanonically), where
$\HH_0(\Sym^n A_\gamma)^*$ has degree $n$, and $t$ has degree $1$ and
weight $-d$.\footnote{In general, there is a canonical map from the RHS
  to the LHS as degree-graded algebras, but it is \emph{not} an
  isomorphism, nor a map of weight-graded algebras; see Remark \ref{noncanrem}.}
\end{corollary}
Note that, when $\gamma$ is a point of finite order of the elliptic
curve, then $\HH_0(A_\gamma)$ is infinite-dimensional, and the
isomorphism $\HH_0(A_\gamma) \cong \HP_0(\OO_X)$ fails. We expect that these are exactly the countably
many $\gamma$ mentioned in the corollary.

Moreover, we deduce the same result for the filtered
deformations of these elliptic algebras as in \cite{VBsfc,EGdelpezzo}:
\begin{corollary}
  For all but countably many $\gamma$, if $A'_{\gamma}$ is an
  associative algebra with an ascending filtration such that
  $\gr A'_{\gamma} = A_\gamma$, then
  $\gr \HH_0(\Sym^n A'_\gamma) \cong \HH_0(\Sym^n A_\gamma)$ as graded
  vector spaces.
\end{corollary}
\begin{remark}
  The isomorphisms above don't have anything to do with the specific
  quantizations.  Generally, the $n$-th symmetric power of any
  formal or generic filtered (or graded) quantization
  $A'$ of $\OO_X$ has
  $\HH_0(\Sym^n A') \cong \HP_0(\Sym^n \OO_X) \otimes K$, as filtered
  (or graded) vector spaces, where $K$ is the base field for the
  deformation (i.e., $K=\C((\hbar))$ for a formal deformation, and $K=\C$ 
  for a generic enough point of an actual deformation).
  Similarly, for any formal or generic filtered/graded Poisson
  deformation $B'$ of $\OO_X$, we have
  $\HP_0(\Sym^n B') \cong \HP_0(\Sym^n \OO_X) \otimes K$, as
  filtered/graded vector spaces.  Note that the semiuniversal formal
  deformations were classified in \cite{EGdelpezzo}, and the actual elliptic
  deformations were deduced as well (these were studied earlier by a
  different method in \cite{VBsfc}).
\end{remark}

\subsection{General symmetric products and
  Poisson-invariant functionals}
It is convenient to replace Poisson homology with invariant
functionals, as follows.  Let $X$ be an affine Poisson variety. Let
$\g_X \subseteq \Gamma(X, TX)$ be the Lie algebra of Hamiltonian vector fields,
which is the same as $\OO_X/Z(\OO_X)$, viewing $\OO_X$ as a Lie algebra, and 
denoting by $Z(\OO_X)$ the Poisson center of $\OO_X$.
As is standard, invariants of a
Lie algebra action of $\g$ on $A$ are denoted by
$A^\g := \{a \in A: g(a)=0, \forall g \in \g\}$.

We now reformulate the problem of computing Poisson homology of the varieties $X^{(n)}$:
\begin{proposition}
  For every affine Poisson variety $X$, there is a canonical graded
  algebra isomorphism
\begin{equation}
\bigoplus_{n \geq 0} \HP_0(X^{(n)})^* \cong \C[\OO_X]^{\g_X} = \bigoplus_{n \geq 0} (\Sym^n (\OO_X)^*)^{\g_X}.
\end{equation}
\end{proposition}
Here, $\C[\OO_X]$ denotes the polynomial functions on the
(infinite-dimensional) vector space $\OO_X$, which is equipped with
the coadjoint action of $\g_X$. The proof is easy and short:
\begin{proof}
We have $HP_0(X^{(n)}) = (\Sym^n \OO_X)/\{\Sym^n \OO_X, \Sym^n \OO_X\}$.  We claim that there is an identification
\begin{equation}
\{\Sym^n \OO_X, \Sym^n \OO_X\} = \g_X (\Sym^n \OO_X).
\end{equation}
Note that we have an obvious inclusion of vector spaces, $\OO_X \subseteq \Sym^n \OO_X$, given by $f \mapsto f \& 1 \& \cdots \& 1$.  
Using this, the inclusion $\supseteq$ above follows from the equality
\begin{equation}
g(f_1 \& f_2 \& \cdots \& f_n) = \{g, f_1 \& \cdots \& f_n\}.
\end{equation}
For the inclusion $\subseteq$, we use the fact that $\Sym^n \OO_X$ is
generated, as a commutative algebra, by the subspace
$\OO_X \subseteq \Sym^n \OO_X$: this follows inductively on $n$ by an
easy argument.  Then, we use the fact that, for any Poisson algebra
$A$ which is generated as a commutative algebra\footnote{In fact,
  $\{A, A\} = \{V, A\}$ assuming only that $A$ is generated as a
  Poisson algebra by $V$, using \eqref{bpid} and the Jacobi identity.}
by $V \subseteq A$, we have $\{A, A\} = \{V, A\}$, by the identity
\begin{equation} \label{bpid}
\{ab, c\} = \{a, bc\} + \{b, ca\}. \qedhere
\end{equation}
\end{proof}

\subsection{A $\C^*$-equivariant vector bundle on $\PP^1$}
We return to the surface $X = Z(Q)$ from \S \ref{mrs}.  In
the Kleinian $A_{m-1}$-case, we will give a short and self-contained
proof of Theorem \ref{mt} in \S \ref{ampfsec} (which would deserve
mention even if the proof below extended to this case).

The main step of the proof for all other cases is to write
$\C[\OO_X]^{\g_X}$ as the algebra of global sections of a certain
infinite-dimensional $\C^*$-equivariant vector bundle on $\PP^1$,
whose definition and structure we explicitly describe in this section.

Henceforth, we assume that we are not in the Kleinian type $A_{m-1}$
case. This has the following important consequence:
\begin{lemma}
  Suppose that $X$ is not of Kleinian type $A_{m-1}$. Then, all nonzero
  homogeneous Hamiltonian vector fields $\xi_f$ have positive degree.
  In particular, $\langle x, y \rangle \cap \{\OO_X, \OO_X\} = 0$.
\end{lemma}
\begin{proof}
  In the non-Kleinian case, it is clear that $\xi_f$ has positive
  degree for all noncentral $f$, since the Poisson bracket has
  nonnegative degree and $\C$ is central. In the Kleinian case,
  looking at \eqref{amdesc}--\eqref{e8desc}, only in the type
  $A_{m-1}$ case is there a Hamiltonian vector field of nonpositive
  degree (in particular, $\xi_x$ has zero degree, and in type $A_1$,
  also $\xi_y$ and $\xi_z$). Then, since $a \leq b \leq c$, for all
  homogeneous $f,g$ with $\{f,g\} \neq 0$, we have
  $|\{f,g\}| \geq |\xi_x(y)| > b$, which implies the final statement.
\end{proof}
Next, for any $m \geq 0$, let $(\OO_X)_m$ denote the subspace of $\OO_X$ of
weighted degree $m$.  It is convenient to consider, rather than
$\C[\OO_X]$, the subalgebra
\begin{equation}\label{fgdefn}
\Fg(X) := \Sym (\bigoplus_{m \geq 0}((\OO_X)_m)^*)
\end{equation}
The entire
algebra $\C[\OO_X]$ is the completion of $\Fg(X)$ by the weight
grading.  In other words, $\Fg(X)$ is the algebra of continuous
polynomial functions on the completion $\OO_{\widehat X} = \C[[x,y,z]]/(Q)$ with respect to the
weight grading. We may view $\OO_{\widehat X}$ as a pro-scheme (with limit
taken over finite-dimensional affine spaces), and in this sense,
$\Fg(X) = \C[\OO_{\widehat X}]$.  Note that $\C[\OO_X]^{\g_X}$ is also a
completion of $\Fg(X)^{\g_X}$. In our case, in fact,
$(\Sym^n(\OO_X)^*)^{\g_X}$ will turn out to be finite-dimensional for
each $n$, and hence $\C[\OO_X]^{\g_X} = \C[\OO_{\widehat X}]^{\g_X}$.

Let
$V := \langle x, y \rangle \subset \OO_{\widehat X}$.  Fix a graded complement  $\OO_{\widehat X}^0$ to $V$ containing all Poisson brackets. Thus,
$\OO_{\widehat X} = V \times \OO_{\widehat X}^0$.
Note that functionals in $\Fg(X)$ are the same as
regular functions on $(V \setminus \{0\}) \times \OO_{\widehat X}^0$.
\begin{lemma} \label{vbdlelemma} The invariants $\Fg(X)^{\g_X}$ can be
  noncanonically identified with regular functions on the total space
  of the pro-vector bundle $Y'$ on $V \setminus \{0\}$ with fiber over
  $(\alpha,\beta)$ given by
\begin{equation} \label{vbdfn} Y'_{(\alpha,\beta)} = \OO_{\widehat X}^0 /
  \{\alpha x+ \beta y, \OO_{\widehat X}\}.
\end{equation}
\end{lemma}
The lemma will be proved in Section \ref{mtpfsec}.  Let us
explain why \eqref{vbdfn} indeed defines a pro-vector bundle.  Note
that $Y'$ is a pro-coherent sheaf which is pulled back from
$\PP^1$. Next, viewing $\OO_{\widehat X}$ as a constant pro-vector
bundle, we may view $Y'$ as the cokernel of the pro-coherent sheaf
map,
 \begin{equation}\label{precydefn}
   \OO_{\widehat X} \otimes \OO(-1) \rightarrow \OO_{\widehat X}^0, \quad f \otimes (\alpha x + \beta y) \mapsto \{f, \alpha x + \beta y\}.
\end{equation}
This map descends to 
\begin{equation} \label{desprecydefn}
\OO_{\widehat X}/\C[[\alpha x + \beta y]] \otimes \OO(-1) \rightarrow \OO_{\widehat X}^0,
\end{equation}
where $\OO_{\widehat X}/\C[[\alpha x + \beta y]]$ is the
quotient of $\OO_{\widehat X}$ by the sub-pro-vector bundle
$\C[[\alpha x + \beta y]] \cong \prod_{i \geq 0} \OO(-1)^{\otimes
  i}$.  

We claim that \eqref{desprecydefn} is injective on fibers.  This
follows by computing that $\C[[\alpha x + \beta y]]$ is the kernel of
$\{\alpha x + \beta y, -\}$, see Lemma \ref{kerbr}.  Hence, this is a
pro-vector bundle map, and the cokernel, $Y'$, is indeed a pro-vector
bundle.

Next, note that $Y'$ is equipped with a $\C^*$-equivariant structure with respect to the $\C^*$-action on $V$, given by, for $w \in \C^*$, 
\begin{equation}
x \mapsto w^{a} x, \quad y \mapsto w^{b} y.
\end{equation}
The action on coordinate functions then has the form
\begin{equation}
\alpha \mapsto w^{-a} \alpha, \quad \beta \mapsto w^{-b} \beta.
\end{equation}

Furthermore, $Y'$ is pulled back from a pro-vector bundle $Y$ on
$\PP^1$. We may thus regard $\Fg(X)^{\g_X}$ as the regular functions
on the total space of the pro-vector bundle
\begin{equation} \label{edfn}
E := Y \oplus \OO(-1)
\end{equation}
on $\PP^1$.  These pro-bundles are also $\C^*$-equivariant.

Note that representations $W$ of $\C^*$ may also be viewed as graded
vector spaces, with action of $w$ in degree $m$ by multiplication by
$w^m$. Thus, we will use the notation $h(W;t)$ for the character of
$W$ viewed as a representation of $\C^*$, i.e., the Hilbert series of
$W$ where $W$ is viewed as a graded vector space (rather than a vector
space with $\C^*$-action).

Next, we describe the structure of $Y$, which will imply the main
theorem. First, recall the following basic facts about
$\C^*$-equivariant vector bundles on $\PP^1$:
\begin{definition}
  Let $\OO(n)_m$ denote $\OO(n)$ with the $\C^*$-equivariant
  structure given by the action of $w \in \C^*$ on the fiber over
  $(1,0)$ as multiplication by $w^m$.
\end{definition} 
In particular, the tautological line bundle $\OO(-1)$ (which appeared
in \eqref{edfn}) is the equivariant bundle $\OO(-1)_{a}$.  We will
need the following well-known result, whose proof is easy and omitted:
\begin{theorem}\label{equivarthm} Let $\PP^1$ be equipped with the above $\C^*$-action.
\begin{enumerate}
\item[(i)] Up to isomorphism, any $\C^*$-equivariant vector bundle on $\PP^1$ has a unique decomposition as a sum of line bundles of the form $\OO(n)_m$.
\item[(ii)] For $n \geq 0$,
\begin{equation}
h(\Gamma(\PP^1, \OO(n)_m);t) = t^m (1 + t^{a-b} + \cdots + t^{n(a-b)}).
\end{equation}
\end{enumerate}
\end{theorem}
\begin{remark}
  In fact, we will work also with pro-$\C^*$-equivariant vector
  bundles, but only those for which the weight-$m$ subspaces of the
  fibers at $(0,1)$ and $(1,0)$ are finite-dimensional for all
  $m \in \ZZ$.  In this case, the above theorem still applies, except
  that now the pro-bundles will be a direct product of possibly
  infinitely many $\OO(n)_m$ (but only finitely many for each value of
  $m$). In particular, the Hilbert series of global sections makes
  sense.
\end{remark}
We may therefore make the following definition:
\begin{definition}
For any $\C^*$-equivariant vector bundle $U$ on $\PP^1$ of the form
$U \cong \bigoplus_{i} \OO(p_i)_{q_i}$, write
\begin{equation}
\chi_{\C^*}(U) = \sum_i s^{p_i} t^{q_i}.
\end{equation}
\end{definition}
We extend this notation in the obvious way to pro-$\C^*$-equivariant vector bundles whose fibers over $(0,1)$ and $(1,0)$ have finite-dimensional weight-$m$ subspaces for all $m \in \ZZ$.
Now, we may state the main technical result of the paper, which implies Theorem \ref{mt}. It will be convenient to use the pro-bundle
\begin{equation}
\widetilde Y := Y \oplus \OO(0)_{a} \oplus \OO(0)_{b}.
\end{equation}
\begin{theorem}\label{mtt}
We have
\begin{equation}\label{chiyfla}
\chi_{\C^*}(\widetilde Y)=\frac{(1-t^{d-a})(1-t^{d-b})(1-t^{d-c})}{(1-t^a)(1-t^b)(1-t^c)(1-t^{d-a}s)}
\end{equation}
\end{theorem}
The next section is devoted to the proof of this theorem.

\section{Proof of Theorem \ref{mtt}}
The following lemma will be a cornerstone of the proof:
\begin{lemma} \label{kerbr} The kernel of
$\{\alpha x + \beta y, -\}: \OO_{\widehat X} \rightarrow \OO_{\widehat X}$
is $\C[[\alpha x + \beta y]] \subset \OO_{\widehat X}$. 
\end{lemma}
(Note that the inclusion
$\C[[\alpha x + \beta y]] \subset \OO_{\widehat X}$ makes sense since, e.g., $Q \notin \C[[x,y]]$).
\begin{proof}
  We first claim that it is sufficient to consider the case where
  either $\alpha = 0$ or $\beta = 0$.  First, if $a=b$ (i.e.,
  $|x|=|y|$), then we may change bases to replace $\alpha x + \beta y$
  with $x$. If $a < b$, then, letting $Z_f$ denote the Poisson
  centralizer of $f$, we have
  $\gr Z_{\alpha x+ \beta y} \subseteq Z_x$ when $\alpha \neq
  0$.
  Since $\C[[\alpha x + \beta y]] \subseteq Z_{\alpha x + \beta y}$,
  when $\alpha \neq 0$ it is sufficient to assume that $\beta = 0$.

  Suppose $\beta = 0$. Let $f\in \OO_{\hat X}$ be such that
  $\{ x, f \}=0$. We need to show that $f$ is a power series in $x$.
  Clearly, one may assume without loss of generality that $f$ is a
  polynomial. Because $X$ is generically symplectic, $f$ must be
  algebraically dependent on $x$. But it is easy to show that $\C[x]$
  is algebraically closed in $\C[x,y,z]/(Q)$ (e.g., any homogeneous
  element in the algebraic closure of $\C[x]$ would have to be a
  rational power of $x$, and only nonnegative integer powers of $x$
  occur in $\C[x,y,z]/(Q)$), so $f\in \C[x]$. The case where $\alpha = 0$ is similar.
\end{proof}

Now, in view of Theorem \ref{equivarthm}, to prove Theorem \ref{mtt}
it suffices to compute the character (Hilbert series) of the vector
spaces
\begin{equation} 
V_q := \Gamma(\PP^1, \widetilde Y \otimes \OO(-q)_0).
\end{equation}
Since $\widetilde Y$ is a quotient of a trivial pro-bundle, it
suffices to take $q \geq 0$.  We obtain the following, which, together
with Theorem \ref{equivarthm}.(ii),
immediately implies Theorem \ref{mtt}:
\begin{proposition} \label{mttprop} For $q \geq 0$,
\begin{equation} \label{vqhilb}
h(V_q;t) = \frac{t^{q(d-a)}(1-t^{d-c})}{(1-t^a)(1-t^b)(1-t^c)}.
\end{equation}
\end{proposition}
\begin{proof}
  We may identify $V_q$ with the space of global sections of
  $\widetilde Y$ which vanish to order $q$ at $(0,1)$, using
  the injections of sheaves,
  $\OO(n)_m \otimes \OO(-1)_0 \cong \OO(n-1)_m \into \OO(n)_m$, which,
  on global sections, are the inclusions of sections vanishing at
  $(0,1)$.\footnote{The map to sections vanishing at $(1,0)$ comes
    from $\OO(n-1)_{m+a-b} \into \OO(n)_m$, so we could have instead
    used sections vanishing to order $q$ at $(0,1)$, but with weights
    shifted by $q(a-b)$. The argument goes through in the same way, swapping
    $x$ with $y$ and $a$ with $b$.}

To prove the proposition, we first rewrite the condition of vanishing
to order $q$ at $(0,1)$, by explicitly describing the
subspace $\{\alpha x + \beta y, \OO_{\widehat X}\}$. As in the introduction,
let $f_x, f_y, f_z$ denote the partial derivatives of
$f \in \C[[x,y,z]]$ with respect to $x, y,$ and $z$.  We have
$\{x,y\} = Q_z$, $\{y,z\} = Q_x$, and $\{z,x\} = Q_y$. Thus,
\begin{equation} \label{brfla}
  \{\alpha x + \beta y, f\} = \alpha (Q_z f_y - Q_y f_z) -
  \beta (Q_z f_x - Q_x f_z).
\end{equation}
In particular, we deduce that
\begin{equation}
  \{\alpha x + \beta y, \C[[x,y]]\} = \C[[x,y]] Q_z.
\end{equation}

We may use this to compute the global sections of $\widetilde Y$.  Let $T$ be a
graded complement to $\C[[x,y]]$ in $\OO_{\widehat X}$, so that
$\OO_{\widehat X} = \C[[x,y]] \oplus T$.  By Lemma \ref{kerbr},
the kernel of $\{\alpha x+ \beta y, -\}$ lies in $\C[[x,y]]$.  Hence, we have an
exact sequence
\begin{equation}
0 \rightarrow \OO(-1)_{a} \otimes T \rightarrow \OO^0_{\widehat X}/ \C[[x,y]]Q_z \rightarrow Y \rightarrow 0,
\end{equation}
and since $Y$ is torsion-free, we conclude by taking global sections that 
\begin{equation} \label{gamyfla}
\Gamma(\PP^1, Y) \cong \OO_{\widehat X}^0/ \C[[x,y]]Q_z.
\end{equation}

To describe $V_q$ for $q > 0$, it will be convenient to
sometimes work in the larger ring $\C[[x,y,z]][Q_z^{-1}]$, and define
the operators
\begin{equation} \label{dxydefn}
  D_x := \partial_x - \frac{Q_x}{Q_z} \partial_z, \quad D_y := \partial_y - \frac{Q_y}{Q_z} \partial_z.
\end{equation}
These operators make sense on $\C[[x,y,z]][Q_z^{-1}] / (Q)$ since
$D_x = -\frac{1}{Q_z} \ad y$ and $D_y = \frac{1}{Q_z} \ad x$ (or
because $D_x(Q) = 0 = D_y(Q)$). Moreover, it is clear that, if we
think of $z$ as implicitly dependent on $x$ and $y$ via $Q = 0$, then
$D_x$ and $D_y$ are the derivatives with respect to $x$ and $y$.  In
particular, on $\C[[x,y]]$, $D_x$ and $D_y$ restrict to the usual
derivative with respect to $x$ and $y$.  Finally, we have
\begin{equation}
[D_x, D_y] = 0,
\end{equation}
which follows from the above (one may also directly compute
that $[D_x, D_y](z) = 0$).

It now follows from \eqref{brfla}, \eqref{gamyfla}, and \eqref{dxydefn}
that $V_q$ is identified with the solutions $G \in \OO_{\widehat X}$ (modulo the subspace $\C[[x,y]]Q_z$) 
of the equations
\begin{multline} \label{fieqns} 
\exists F_1, \ldots, F_q \in  \C[[x,y,z]]/(Q) \text{ s.t. }  
G = Q_z D_x F_1, \\
D_y F_1 = D_x F_2, \ldots, 
D_y F_{q-1} = D_x F_q.
\end{multline}

We break up most of the rest of the proof into lemmas.  It suffices to consider \emph{homogeneous} solutions to the above equations, which we do from now on. In particular, this means we can (and will) work in the uncompleted rings $\C[x,y,z], \C[x,y,z]/(Q), \C[x,y,z][Q_z^{-1}]/(Q)$, etc.
\begin{lemma}
Every homogeneous solution to \eqref{fieqns} has the form
\begin{equation} \label{figheqns}
F_i = D_y^{i-1} D_x^{q-i} H, \quad G = Q_z D_x^q H,
\end{equation}
for some homogeneous element $H \in \C[x,y,z]/(Q)$, which is
uniquely determined by the $F_i$.
\end{lemma}
The proof mimics a standard proof of the
Poincar\'e lemma.
\begin{proof}
We show, inductively on $j$, that there exist unique homogeneous
$H_{i,j}$, for $i+j \leq q$ and $i \geq 1$, $j \geq 0$, 
such that (for $j \geq 1$)
\begin{equation}
D_x H_{i,j} = H_{i,j-1}, \quad D_y H_{i,j} = H_{i+1, j-1}, \quad H_{i,0} := F_i.
\end{equation}
Then, it follows that $H = H_{1,q-1}$ has the desired property.

To do this, we use the formula, valid for all homogeneous $f$ of
positive degree:
\begin{equation}
|f| f = ax D_x f + by D_y f.
\end{equation}
Thus, given any $g$ and $h$, there exists $f$ such that $D_x f = g$
and $D_y f = h$ if and only if $D_y g = D_x h$, and in this case,
$|f| f = ax g + ay h$.  The inductive step therefore follows by
setting $g = H_{i,j}$ and $h = H_{i+1,j}$.
\end{proof}
Next, we have to find what possible $H$ can arise. This is answered by
\begin{lemma}
  Let $H \in \C[x,y,z]$ be homogeneous.  Then, the following
  are equivalent:
\begin{enumerate}
\item[(i)]  For all polynomials $f$ of degree $\leq n$, we have
\begin{equation} \label{hz1}
f(D_x, D_y) H \in \C[x,y,z] + Q \cdot \C[x,y,z][Q_z^{-1}].
\end{equation}
\item[(ii)] We have
\begin{equation} \label{hz2}
H_z \in (Q^n) + (Q)_z.
\end{equation} 
\end{enumerate}
\end{lemma}
Here and below, $(Q)_z$ is the partial derivative of the
\emph{ideal} $(Q) \subset \C[x,y,z]$, not the element.  

As a consequence of the lemma, we deduce that the possible $H$ in
\eqref{figheqns} are exactly those satisfying \eqref{hz2} for
$n = q-1$.
\begin{proof}
  The implication (ii) $\Rightarrow$ (i) is easy: since $D_x, D_y$ are
  well-defined on $\C[[x,y,z]][Q_z^{-1}] / (Q)$, given \eqref{hz2}, we may
  assume that $Q^n \mid H$, and \eqref{hz1} follows immediately.

  Next, we prove (i) $\Rightarrow$ (ii) inductively on $n$.  For $n=0$
  the assertion is vacuous.  Since the assertion does not depend on
  the choice of $H$ modulo $(Q)$, we may assume inductively that
  $Q^{n-1} \mid H_z$ and \eqref{hz1} holds. We must prove that, up
  to adding an element of $(Q)$ to $H$, we have $Q^n \mid H_z$.

Letting $f(D_x, D_y) = D_x$ in \eqref{hz1}, we have
\begin{equation}
H_x - \frac{Q_x}{Q_z} H_z \in \C[x,y,z] + Q \cdot \C[x,y,z][Q_z^{-1}].
\end{equation}
It follows that
\begin{equation}
Q_x H_z \in Q_z \C[x,y,z] + Q \cdot \C[x,y,z][Q_z^{-1}],
\end{equation}
but since also $H_z \in \C[x,y,z]$, we in fact have
\begin{equation}
Q_x H_z \in (Q, Q_z).
\end{equation}
Similarly, using $f(D_x, D_y) = D_y$ in \eqref{hz1}, we have
\begin{equation}
Q_y H_z \in (Q,Q_z).
\end{equation}
By the following Lemma \ref{alglemm}, $H_z \in (Q, Q_z)$ itself.  This proves \eqref{hz2} in the case $n=1$.  

We proceed now under the assumption that $n \geq 2$.  Write $H_z = Q^{n-1} h$. Then 
\begin{equation}
D_x H  - H_x \in Q \cdot \C[x,y,z][Q_z^{-1}].
\end{equation}
As a consequence, \eqref{hz1} implies that
\begin{equation}
f(D_x, D_y) H_x \in \C[x,y,z] + Q \C[x,y,z][Q_z^{-1}],
\end{equation}
for all polynomials $f$ of degree $\leq n-1$. By the inductive hypothesis
applied to $H_x$, we conclude that
\begin{equation}
H_{xz} \in (Q^{n-1}) + (Q)_z.
\end{equation}
Substituting $H_z = Q^{n-1} h$, we find that
\begin{equation}\label{qxhz1}
(n-1)Q^{n-2} Q_x h \in (Q^{n-1}) + (Q)_z.
\end{equation}
Next, note that 
\begin{equation} \label{qzqn}
(Q)_z \cap (Q^{n-2}) = (Q^{n-1})_z := \{g_z \mid g \in (Q^{n-1})\},
\end{equation}
 since if $Q \nmid g$, we have $\partial_z (Q^j g) = j Q^{j-1} Q_z g + Q^j g_z$, which is in $(Q^{j-1})$ but not $(Q^j)$. Applying this to \eqref{qxhz1}, we deduce that
\begin{equation}
Q^{n-2} Q_x h \in (Q^{n-1}) + (Q^{n-1})_z = Q^{n-2} (Q, Q_z).
\end{equation}
Dividing by $Q^{n-2}$, we get that
\begin{equation}
Q_x h \in (Q, Q_z),
\end{equation}
and similarly,
\begin{equation}
Q_y h \in (Q, Q_z).
\end{equation}
Thus, again applying Lemma \ref{alglemm} below, we find that
$h \in (Q, Q_z)$, and hence
\begin{equation}
H_z \in Q^{n-1} (Q, Q_z) \subseteq (Q^n) + (Q)_z. \qedhere
\end{equation}
\end{proof}
The above proof rested on the following basic result:
\begin{lemma}\label{alglemm}
If $f \in \C[x,y,z]$ satisfies
\begin{equation}\label{qxycond}
Q_x f \in (Q, Q_z), \quad Q_y f \in (Q, Q_z),
\end{equation}
then it follows that
\begin{equation}
f \in (Q, Q_z).
\end{equation}
\end{lemma}
\begin{proof}
  We claim that \eqref{qxycond} implies that the ideal
  $(f) \subseteq \C[x,y,z]/(Q,Q_z)$ is a torsion module supported at
  the origin.  Since (as we will recall), such torsion modules cannot
  be submodules of $\C[x,y,z]/(Q,Q_z)$, we will deduce that $f=0$.

  To prove the claim, note that, since the singularity at the origin
  is isolated, at every closed point in $Z(Q,Q_z)$ other than the
  origin, either $Q_x$ or $Q_y$ must be nonvanishing. Hence, in every
  local ring other than at the origin, either $Q_x$ or $Q_y$ is a
  unit. Thus, $f$ is zero in every local ring other than the origin,
  i.e., $(f)$ is a torsion module supported at the origin.

  Next, note that, since $Q$ is irreducible, $\C[x,y,z]/(Q)$ is a
  domain, and hence $Q, Q_z$ form a regular sequence in $\C[x,y,z]$.
  That is, we have a Koszul resolution of $\C[x,y,z]/(Q,Q_z)$ of
  length two. Since any torsion module $M$ supported at a
  point satisfies $\Ext^i(M, \C[x,y,z]) = 0$ for $i < 3$, 
  the long exact sequence of cohomology implies that
  $\Ext^i(M, \C[x,y,z]/(Q,Q_z)) = 0$ for $i < 1$.  Thus,
  $\Hom((f), \C[x,y,z]/(Q,Q_z)) = 0$.  This implies that $f = 0$, as
  desired.
\end{proof}
To complete the proof of the proposition, first note that solutions to
\eqref{fieqns} such that $G \in \C[x,y] Q_z$, which form the subspace
we wanted to quotient by, are exactly those for which the
$F_i \in \C[x,y]$, and hence also $H \in \C[x,y]$.  Therefore, to
compute the Hilbert series of $V_q$, it remains, for each degree
$|G| = m$, to find the dimension of the space of homogeneous 
$H$ of degree $|G|+qa-(d-c)$ such that $Q^{q-1} \mid H_z$, modulo the
space of such $H$ which are polynomials in $x$ and $y$.  In other
words, we seek the dimension of the space of elements $H_z$ of degree
$|G|+qa-d$ that are multiples of $Q^{q-1}$, modulo $(Q)_z$.  By
\eqref{qzqn}, this is equivalent to considering
$(Q^{q-1}) = \{Q^{q-1} g\mid g \in \C[x,y,z]\}$ modulo
$(Q^q)_z = \{Q^{q-1} (Q f_z + q Q_z f) \mid f \in \C[x,y,z]\}$.  That
is, writing $H_z = Q^{q-1} g$, our problem reduces to considering the
space of polynomials $g$ of degree $|g| = |G|+q(a-d)$ modulo
$\langle Q f_z + q Q_z f \mid |f|=|g|+(c-d) \rangle$.  We conclude
that the Hilbert series of $V_q$ is exactly \eqref{vqhilb}.
\end{proof}

\section{Proof of Theorem \ref{mt} when $X$ is not of type
  $A_{m-1}$}
\label{mtpfsec} We first consider the case where $X$ is not a
singularity of type $A_{m-1}$. The main step left is to give the
promised proof of Lemma \ref{vbdlelemma}.
\begin{proof}[Proof of Lemma \ref{vbdlelemma}]
Let $\mathcal{G}_X = \exp(\g_X)$ be the group of Poisson
automorphisms of $\OO_{\widehat X}$ generated by the flow of Hamiltonian vector
fields $\xi_f$ for $f \in \OO_{\widehat X}$ (so, $\xi_f|_g = \{f,g\}$, identifying
the tangent space at every point of $\OO_{\widehat X}$ with $\OO_{\widehat X}$ itself).
It is clear that
$\Fg(X)^{\g_X} = \Fg(X)^{\mathcal{G}_X} = \C[(V \setminus \{0\})
\times \OO_{\widehat X}^0]^{\mathcal{G}_X}$.

The idea behind the proof of the lemma is to view the fibers of $Y$ as slices
in $\alpha x+ \beta y+\OO_{\widehat X}^0$ to the orbits of the group
$\mathcal{G}_X$, in the following sense.

For a fixed $(\alpha, \beta)$, note that
$\alpha x + \beta y + \OO_{\widehat X}^0$ is stable under
$\mathcal{G}_X$ (since $\OO_{\widehat X}^0$ contains all Poisson
brackets).  Hence, we have a map
$\OO_{\widehat X}/\mathcal{G}_X \rightarrow V$ with a canonical zero
section
$V := \langle x,y\rangle \subseteq \OO_{\widehat X}/\mathcal{G}_X$.
Let $U'$ be the pro-bundle over $V \setminus 0$ whose fiber at
$\alpha x + \beta y$ is the tangent space to the fiber of the above
map.  In other words,
\begin{equation}
U'_{\alpha x + \beta y} = T_{\alpha x + \beta y} ((\alpha x + \beta y + \OO_{\widehat X}^0) / \mathcal{G}_X).
\end{equation}
The pro-bundle $U'$ is evidently pulled back from a pro-bundle on $\PP^1$. Call this $U$.
\begin{claim} \label{isoclaim} 
\begin{enumerate}
\item[(i)] The punctured plane
  $V \setminus \{0\} \subseteq \OO_{\widehat X}/\mathcal{G}_X$
  consists of smooth points.
\item[(ii)] We have a canonical isomorphism of pro-bundles
\begin{equation} \label{isomap}
U \cong Y.
\end{equation}
\end{enumerate}
\end{claim}
We will prove this claim below. For now, we assume it.  Introduce the
filtration on $\C[\OO_{\widehat X}/\mathcal{G}_X]$ by powers of the
ideal $I_V$ of functions vanishing on the plane $V$.  By the claim,
$V \setminus \{0\}$ consists of smooth points, and hence
\begin{equation}
\gr_{I_V} \C[\OO_{\widehat X}/\mathcal{G}_X] = \C[U'] \cong \C[Y'],
\end{equation}
where the latter denotes the global functions on the total space of
the pro-bundle $Y'$.  The total space of $Y'$ (a pro-bundle over $V \setminus \{0\}$), is the same as the total space of $E = Y \oplus \OO(-1)_{a}$ (a pro-bundle over $\PP^1$), and we deduce that
\begin{equation}
\gr_{I_V} \C[\OO_{\widehat X}/\mathcal{G}_X] \cong \C[E].
\end{equation}

By Theorem \ref{mtt}, $\C[E]$ is in fact a polynomial
algebra on homogeneous generators, finitely many in each degree. 
 Hence, the lemma follows from 
 \begin{claim} \label{grclaim} Let $A$ be a graded commutative algebra
   $A$ with a descending graded filtration
   $A = F_0 A \supseteq F_1 A \supseteq \cdots$ with $\cap_i F_i A = 0$ such
   that $\gr A \cong \Sym W$, where $W$ is a bigraded vector space which
   is finite-dimensional in each bidegree. Then,
\begin{equation}\label{symwiso}
A = \Sym \widetilde W \cong \gr A,
\end{equation}
for any graded lifting $\widetilde W$ of $W$ to $A$.  
\end{claim}
\begin{proof}
  We have a canonical morphism of algebras,
  $\iota: \Sym \widetilde W \to A$, which becomes an isomorphism when we
  take associated graded.  Hence, it must be a monomorphism. To prove surjectivity, fix a degree $n \geq 0$. The surjectivity of $\Sym W \rightarrow \gr A$ says that
\begin{equation} \label{grsurj}
\iota(\Sym \widetilde W)_n + (F_{i+1} A)_n \supseteq (F_i A)_n, \forall i.
\end{equation}
Since $\cap_i F_i A = 0$ and each $(F_i A)_n$ is finite-dimensional,
there must exist $j \geq 0$ such that $(F_{j+1} A)_n = 0$. We deduce from \eqref{grsurj} that
$\iota(\Sym \widetilde W)_n \supseteq (F_j A)_n$, and applying
\eqref{grsurj} $j$ more times, we deduce that
$\iota(\Sym \widetilde W)_n \supseteq (F_0 A)_n = A_n$.
\end{proof}
We apply this for $A = \C[\OO_{\widehat X}/\mathcal{G}_X]$, with the filtration by  powers of $I_V$, and $A_0 = \C[E]$.
\end{proof}

\begin{proof}[Proof of Claim \ref{isoclaim}]
Since $\OO_{\hat X}$ is a
$\mathcal{G}_X$-representation, for any element $f \in \OO_{\hat X}$, we have
\begin{equation}
T_{\alpha x + \beta y + \hbar f} (\OO_{\hat X}/\mathcal{G}_X) = \{\alpha x + \beta y + \hbar f, \OO_{\hat X}\}.
\end{equation}
It suffices to show that the RHS is saturated as a $\C[[\hbar]]$-module, i.e.,
if 
\begin{equation}\label{nonsat}
\{\alpha x + \beta y + \hbar f, g\} \in \hbar \OO_{\hat X}[[\hbar]],
\end{equation}
then
\begin{equation} \label{nonsat2}
g \in Z_{\alpha x + \beta y + \hbar f} + \hbar \OO_{\hat X}[[\hbar]],
\end{equation}
where $Z_{\alpha x + \beta y + \hbar f}$ is the Poisson centralizer.
Since $Z_{\alpha x + \beta y} = \C[[\alpha x + \beta y]]$,
\eqref{nonsat} can only hold if
$g \in \C[[\alpha x + \beta y]] + \hbar \OO_{\hat X}[[\hbar]]$.  Then,
using that
$\C[[\alpha x + \beta y + \hbar f]] \subseteq Z_{\alpha x +\beta y +
  \hbar f}$
(in fact, this is an equality), we see that \eqref{nonsat2} holds.
\end{proof}

Now, Theorem \ref{mt} (in the non-type $A$ case) follows from Theorem
\ref{mtt}, since we have identified the regular invariant functions
with the regular functions on the total space of
$E = Y \oplus \OO(-1)_{a}$.  That is, we take the global sections of
$\Sym E^* = \Sym Y^* \otimes \bigoplus_{m \geq 0} (\OO(1)_{-a})^{\otimes m}$. If we decompose
$Y = \prod_i \OO(n_i)_{m_i}$, then the desired regular functions form
a polynomial algebra on the generators
$f_i \in \Gamma(\PP^1,\OO(-n_i)_{-m_i} \otimes (\OO(1)_{-a})^{\otimes
  n_i})
\setminus \{0\}$,
of weight $-m_i-n_i a$ and polynomial degree $n_i+1$ (these vector
spaces are one-dimensional, so any nonzero $f_i$ will work), together
with the generators $\alpha, \beta$, which are the sections of
$\OO(1)_{-a}$.  Letting $L$ again denote the span of these generators, we
see from \eqref{chiyfla} that
\begin{multline}
h(L;t^{-1}) = \sum_{i=0}^{\infty} s^{i+1} t^{ia} t^{i(d-a)} \frac{(1-t^{d-a})(1-t^{d-b})(1-t^{d-c})}{(1-t^a)(1-t^b)(1-t^c)} \\ = \frac{s(1-t^{d-a})(1-t^{d-b})(1-t^{d-c})}{(1-t^a)(1-t^b)(1-t^c)(1-t^d s)} = \frac{h(J_Q;t)s}{1-t^d s},
\end{multline}
using the well-known formula for $h(J_Q;t)$ (which says that
$(Q_x,Q_y,Q_z)$ form a regular sequence in $\C[x,y,z]$). This completes the proof of Theorem \ref{mt} in the non-$A_{m-1}$ case.

\section{Proof of Theorem \ref{mt} in the $A_{m-1}$
  case}\label{ampfsec}
This case involves the Poisson algebra $C[X] = \C[x,y,z]/(Q)$ for
$Q = x^m + y^2 + z^2$.
  It will be convenient to present this slightly differently, as
  $\OO_X = \C[\rx^m, \rx \ry, \ry^m] \subset \C[\rx,\ry]$, with the usual Poisson
  bracket $\{ \rx, \ry \} =1$ (this is the natural presentation from the point
  of view $X = \C^2 / (\ZZ/m)$), and $|\rx|=|\ry|=1$.  As before, define $\Fg(X)$ as in
  \eqref{fgdefn}.  We once again have that
  $\C[\OO_X]^{\mathfrak{g}_X}$ is (a completion of)
  $\Fg(X)^{\mathfrak{g}_X} \cong \Fg(X)^{\mathcal{G}_X}$, where
  $\mathcal{G}_X = \exp(\mathfrak{g}_X)$ is the group of Poisson
  automorphisms of $\OO_{\widehat X} = \C[[\rx^m,\rx \ry,\ry^m]]$ obtained from the flow
  of Hamiltonian vector fields. Also, since it will turn out that
  $\Fg(X)^{\mathcal{G}_X}$ is finite-dimensional in each
  degree,\footnote{This is also a consequence of the known fact that
    $HP_0(\C[\C^{2n}]^{(\ZZ/m)^n \rtimes S_n})$ is finite-dimensional,
    a special case of the result of the appendix to \cite{BEG}, that
    $HP_0(\C[\C^{2n}]^G)$ is finite-dimensional for all finite
    $G < \Sp_{2n}$.}
  in fact $\C[\OO_X]^{\mathfrak{g}_X} \cong \Fg(X)^{\mathcal{G}_X}$ as
  graded vector spaces. Now, Theorem \ref{mt} will follow from the
following replacement for Theorem \ref{mtt}:
  \begin{proposition}\label{kleinprop} The following set is a
    slice to the $\mathcal{G}_X$-orbits in $\C[[\rx^m,\rx \ry,\ry^m]]$ with
    nonzero coefficient of $\ry^m$ or of $\rx \ry$:
\begin{equation}\label{anslice}
\ry^{m} + \C[[\rx^m]] \langle 1, \rx \ry, (\rx \ry)^2, \ldots, (\rx \ry)^{m-2} \rangle. 
\end{equation}
\end{proposition}
\begin{proof}
  If the coefficient of $\ry^m$ is zero but not the coefficient of
  $\rx \ry$, we can apply $e^{\xi_{(\ry^m)}}$ to make the coefficient of
  $\ry^m$ nonzero.  Once we have a nonzero coefficient of $\ry^m$, by
  applying rescalings
  $\ry^m \mapsto \gamma \ry^m, \rx^m \mapsto \gamma^{-1} \rx^m$, we can make
  the coefficient of $\ry^m$ one.  Next, we use the lexicographical
  ordering $\prec$ on monomials $\C[[\rx^m,\rx \ry,\ry^m]]$, where
  $\rx^a \ry^b \prec \rx^{a'} \ry^{b'}$ if either $a < a'$ or $a=a'$ and
  $b < b'$. Note that, for $a \neq 0$ and $f \in \C[[\rx^m,\rx \ry,\ry^m]]$ a power series with zero coefficient of $\ry^m$,
\begin{equation}
\{\rx^a \ry^b, \ry^m + f\} = \{\rx^a \ry^b, \ry^m\} + \ldots, \quad e^{\xi_{(\rx^a \ry^b)}} (\ry^m + f) = \{\rx^a \ry^b, \ry^m\} + \ldots,
\end{equation}
where $\ldots$ denotes higher-order terms with respect to $\prec$.
Hence, by applying elements $e^{\xi_{(\rx^a \ry^b)}}$, we can kill off
all monomials which appear in $\{\ry^m, \C[[\rx^m, \rx \ry,
\ry^m]]\}$,
which is a complement to
$\C[[\rx^m]] \langle 1, \rx \ry, (\rx \ry)^2, \ldots, (\rx \ry)^{m-2}
\rangle$.
Hence, any orbit with nonzero coefficient of either $\rx \ry$ or
$\ry^m$ contains a point of the form \eqref{anslice}.

  It remains to prove that this point is unique.  In other words,
  if $g \in \C[[\rx^m,\rx \ry,\ry^m]]$ satisfies
\begin{equation} \label{xigan}
e^{\xi_g}(\ry^m + f) = \ry^m + f', \quad f,f' \in \C[[\rx^m]] \langle 1, \rx \ry, (\rx \ry)^2, \ldots, (\rx \ry)^{m-2} \rangle,
\end{equation}
then $f=f'$.  It suffices to show that, if $f \neq f'$, then the
lowest-order term in $f-f'$ with respect to $\prec$ lies in
$\{\ry^m, \C[[\rx^m,\rx \ry,\ry^m]]\}$ (since this is impossible).

We assume that $f$ and $f'$ have no constant term.
Note that, for any power series $R(\ry^m+f)$ in $\ry^m+f$, the
operator $e^{\xi_{R(\ry^m+f)}}$ fixes $\ry^m + f$.  Using the
Campbell-Baker-Hausdorff formula, we can replace $g$ by an element
$g'$ such that $e^{\xi_{g'}} = e^{\xi_g} e^{\xi_{R(\ry^m+f)}}$.  
Hence, inductively on $\prec$, we
may assume that the coefficient in $g$ of every monomial $\ry^{km}$ is
zero.  In this case, the lowest-order term in $f-f'$ with respect to
$\prec$ appears in $\{\ry^m, g\}$.
\end{proof}
\begin{corollary}
The invariant regular functions $\Fg(X)^{\mathcal{G}_X}$ restrict isomorphically
to the regular functions on the slice \eqref{anslice}.
\end{corollary}
\begin{proof}
  By the proposition, restriction to the slice \eqref{anslice} identifies
  regular $\mathcal{G}_X$-invariant functions on
  the subvariety $U$ of $\C[[\rx^m,\rx \ry,\ry^m]]$ consisting of power series
  whose coefficient of $\ry^m$ or $\rx \ry$ is nonzero with regular functions on \eqref{anslice}.
  Since $U$ is the complement of an affine subspace of
  codimension two, all regular (invariant) functions on $U$ extend to
  regular (invariant) functions on all of $\C[[\rx^m,\rx \ry,\ry^m]]$.
\end{proof}
It remains to compute the algebra of regular functions on
\eqref{anslice}. This is a polynomial algebra generated by the
coordinate functions $w_{\rx^r \ry^s}$ of the slice, by which we mean that the
point with coordinates $(w_{\rx^r \ry^s})$ in the slice is
\begin{equation}\label{slicecoord}
\ry^m + \sum_{r,s} w_{\rx^r \ry^s} \rx^r \ry^s.
\end{equation}
These do \emph{not} necessarily have degree one as polynomial
functions on $\OO_{\hat X}$.  To determine the degree we may consider
the value of the slice coordinate $w_{\rx^r \ry^s}$ on $\gamma$ times
\eqref{slicecoord} for arbitrary $\gamma \in \C$.  We can compute this
by applying the element of $\mathcal{G}_X$ which rescales by
$\ry^m \mapsto \gamma^{-1} \ry^m$ and $\rx^m \mapsto \gamma \rx^m$.  We deduce
that the degree of $w_{\rx^r \ry^s}$ is $\frac{r-s}{m}+1$.  This yields
exactly \eqref{mtfla}, proving Theorem \ref{mt} in the $A_{m-1}$ case.

\section{Proof of Theorems \ref{kleindeg} and \ref{defdeg} and Corollary \ref{ellcor}}
To prove this, we first need to recall the structure of the zeroth
Hochschild homology of symmetric products of algebras.  Form the \emph{coalgebra}
\begin{equation}
  \HHH(A) := \bigoplus_{n \geq 0} \HH_0(\Sym^n A),
\end{equation}
where the comultiplication map is given by the symmetrization maps
\begin{equation}
\Sym^{p} A \rightarrow \bigoplus_{m+n=p} \Sym^m A \otimes \Sym^n A.
\end{equation}
(When $\HH_0(A)$ is finite-dimensional, this is dual to the symmetrization
maps that we considered earlier.)
We
then have the following result:
\begin{theorem}\cite[Corollary 3.3]{EO} \label{eothm} Let $A$ be an infinite-dimensional simple algebra over $\C$ with trivial center.
Then, the algebra $\HHH(A)$ is a polynomial coalgebra,
\begin{equation} \label{eofla}
\HHH(A) \cong \Sym (\HH_0(A)[t]),
\end{equation}
where the isomorphism is the unique graded coalgebra map (with
$\HH_0(A)$ and $t$ both in degree one) such that the composition with
the projection to $t^{n-1} \HH_0(A)$,
\begin{equation}
\HH_0(\Sym^n A) \rightarrow \Sym (\HH_0(A)[t]) \onto t^{n-1} \HH_0(A),
\end{equation}
has the form $[a_1 \& \cdots \& a_n] \mapsto \frac{1}{n!}  \sum_{\sigma \in S_n}t^{n-1} [a_{\sigma(1)} \cdots a_{\sigma(n)}]$.
\end{theorem}
We note that the above theorem is not stated in quite this way in \cite{EO}, but rather in the equivalent formulation that
\begin{equation}
\HH_0(\Sym^n A) \cong \bigoplus_{\nu \in \mathcal{P}_n} \bigotimes_{i \geq 1} \Sym^{\nu_i} \HH_0(A),
\end{equation}
where $\mathcal{P}_n$ is the set of partitions of $n$, and $\nu_i$
denotes the number of cells of $\nu$ of size $i$.

We have the following immediate corollary:
\begin{corollary}\label{dualcor}
  In the situation of Theorem \ref{eothm}, if $\HH_0(A)$ is
  finite-dimensional, then the commutative algebra
  $\bigoplus_{n \geq 0} \HH_0(\Sym^n A)^*$ is freely generated by the
  vector spaces
  $\langle [f^{\& n}] \mapsto T([f^n]) \rangle_{T \in \HH_0(A)^*}
  \subseteq \HH_0(\Sym^n A)^*$.
\end{corollary}

\begin{proof}[Proof of Theorem \ref{kleindeg}]
  The algebra $\Weyl_{2n}^H$ is well-known to be simple for all finite
  groups $H < \Sp_{2n}$.  Indeed, since $\Weyl_{2n}$ is simple, so is
  the smash-product algebra $\Weyl_{2n} \rtimes H$, and this is
  therefore Morita equivalent to $\Weyl_{2n}^H$.  Obviously,
  $\Weyl_{2n}^H$ is also infinite-dimensional. Since, for finite
  $G< \Sp_2$ and $X = \C^2/G$, it is known that
  $\HH_0(\Weyl_{2}^G) \cong \HP_0(\OO_{X})$, (e.g. by comparing
  \cite{AFLS} for the former with the formulas mentioned in the
  introduction for the latter), the theorem follows from Theorem
  \ref{eothm} and the main Theorem \ref{mt}.
\end{proof}

\begin{proof}[Proof of Theorem \ref{defdeg}]
  Although the deformation quantization of $\OO_X$ is not, in general,
  simple, we may deform $X$ to $Z(Q-\lambda)$, which is symplectic for
  $\lambda \neq 0$.  Let $A_\lambda := \C[x,y,z]/(Q-\lambda)$, and let
  $A_{\lambda, \hbar}$ be its deformation quantization; the algebra
  $A_{\lambda,\hbar}[\hbar^{-1}]$ is simple.  By results of
  Nest-Tsygan \cite{NeTs}, we have
  $\HH_0(A_{\lambda, \hbar}[\hbar^{-1}]) \cong
  \HP_0(A_\lambda)((\hbar))$
  for $\lambda \neq 0$, i.e., the Brylinski spectral sequence
  degenerates. Moreover, the Betti numbers of $A_{\lambda}$ are
  $1, 0,$ and $\mu_Q$, where $\mu_Q = \dim \C[x,y,z]/(Q_x, Q_y, Q_z)$
  is the Milnor number of $X$.  Hence,
  $\dim_{\C((\hbar))} \HH_0(A_{\lambda, \hbar}[\hbar^{-1}]) = \dim
  \HP_0(A_0)$.
  By Theorem \ref{eothm} and Theorem \ref{mt}, we deduce that
  $\HH_0(\Sym A_{\lambda, \hbar}[\hbar^{-1}]) \cong \HP_0(\Sym
  A_0)((\hbar))$
  as graded algebras (with degree $n$ corresponding to $\Sym^n$, so
  not looking at the grading on $A_0$ yet).  However, as
  $\lambda \rightarrow 0$,
  $\HH_0(\Sym^n A_{\lambda, \hbar}[\hbar^{-1}])$ can only increase in
  dimension over $\C((\hbar))$, but the Brylinski spectral sequence
  shows that
  $\dim_{\C((\hbar))} \HH_0(\Sym^n A_{0,\hbar}[\hbar^{-1}]) \leq \dim
  \HP_0(\Sym^n A_0)$
  for all $n$.  Hence, the dimensions are equal, and the Brylinski
  spectral sequence degenerates.
\end{proof}

\begin{proof}[Proof of Corollary \ref{ellcor}] In a formal punctured
  neighborhood of $\gamma = 0$, we see from the above that the zeroth
  Hochschild homology is constant and as above.  Hence, by a standard
  argument, the same is true when $\hbar$ is replaced by actual values
  of $\gamma$ that do not obey a countable number of polynomial
  equations, i.e., for all but countably many $\gamma$, and the zeroth
  Hochschild homology is as above.
\end{proof}
\begin{remark}\label{noncanrem} We deduce from the above that, for the elliptic
  algebras $A_\gamma$, the two sides of \eqref{eofla} are abstractly
  isomorphic as bigraded algebras (Corollary \ref{ellcor}).
  However, the map defined in Theorem \ref{eothm} is \emph{not} an
  isomorphism:  for example, consider $T \in \HH_0(A)^*$ which
  takes the degree-zero coefficient of an element of $A/[A,A]$. Then,
  $T([a^2]) = T([a])^2$, so the map is not an isomorphism
  (cf.~Corollary \ref{dualcor}).  Moreover, the map of Theorem
  \ref{eothm} is not even a bigraded map: while it preserves degree
  (and is thus a graded map), it does not preserve weight: while
  $|t^{n-1} T| = -(n-1)d+|T|$ on the RHS for $T \in
  \HH_0(A_\gamma)^*$,
  the element $[f^{\& n}] \mapsto T([f^n])$ has degree $|T|$ on the
  LHS. So the fact that an isomorphism between the two sides of
  \eqref{eofla} exists is subtle and there may not be a canonical one.
\end{remark}

\section{Acknowledgements}
We thank Qingchun Ren for writing computer programs which allowed us
to compute $\HP_0$ for certain surfaces in low degrees (as part of an
MIT UROP with the authors).  We are grateful to J. Alev for useful
discussions.  The first author is partially supported by NSF grant
DMS-0504847.  The second author is a five-year fellow of the American
Institute of Mathematics.

\bibliographystyle{amsalpha}
\bibliography{master}
\end{document}